\documentclass[11pt]{amsart}
\renewcommand\baselinestretch{1.3}
\usepackage{amsmath, pb-diagram}
\usepackage{amssymb}
\usepackage{amscd}
\begin{document}

\title[Strongly Rad-clean Matrices over Commutative Local Rings]{Strongly Rad-clean Matrices over Commutative Local Rings}
\author{Huanyin Chen}
\address{\emph{Huanyin Chen}\\ Department of Mathematics, Hangzhou Normal University\\
Hangzhou 310036, China} \email {huanyinchen@aliyun.com}

\author{Handan Kose}
\address{\emph{Handan Kose}\\Department of Mathematics, Ahi Evran University\\ Kirsehir,
Turkey} \email{handankose@gmail.com}

\author{Yosum Kurtulmaz$^*$}
\address{\emph{Yosum Kurtulmaz}\\
\small{Department of Mathematics, Bilkent University\\ Ankara,
Turkey}} \email{yosum@fen.bilkent.edu.tr}

\thanks{$^*$Corresponding author: Yosum Kurtulmaz}

\newtheorem{thm}{Theorem}[section]
\newtheorem{lem}[thm]{Lemma}
\newtheorem{prop}[thm]{Proposition}
\newtheorem{cor}[thm]{Corollary}
\newtheorem{exs}[thm]{Examples}
\newtheorem{df}[thm]{Definition}
\newtheorem{nota}{Notation}
\newtheorem{rem}[thm]{Remark}
\newtheorem{ex}[thm]{Example}

\renewcommand\baselinestretch{1.3}

\begin{abstract}
An element $a\in R$ is strongly rad-clean provided that there exists an idempotent
$e\in R$ such that $a-e\in U(R)$,  $ae=ea$ and $eae\in J(eRe)$. In
this article, we investigate strongly rad-clean matrices over a
commutative local ring. We completely determine when a $2\times 2$
matrix over a commutative local ring is strongly rad-clean. Application to the matrices
over power-series is also given.

\vspace{2mm} \noindent {\bf 2010 Mathematics Subject Classification:} 16S50, 16U99.

\end{abstract}
\maketitle

\section{Introduction}
An element $a\in R$ is \emph{strongly clean} provided
that it is the sum of an idempotent and a unit that commutes. A
ring $R$ is \emph{strongly clean} provided that every element in
$R$ is strongly clean. A ring $R$ is local if it has only one
maximal right ideal. As is well known, a ring $R$ is local if and
only if for any $x\in R$, $x$ or $1-x$ is invertible. Strongly
clean matrices over commutative local rings was extensively
studied by many authors from very different view points (cf. [1-3]
and [8]). Recently, a related cleanness of triangular matrix rings
over abelian rings was studied by Diesl et al. (cf. [7]). In fact,
every triangular matrix ring over a commutative local ring is
strongly rad-clean (cf. [6]).

Following Diesl, we say that $a\in R$ is
\emph{strongly rad-clean} provided that there exists an idempotent
$e\in R$ such that $a-e\in U(R),\ ae=ea$ and $eae\in J(eRe)$ (cf.
[6]). A ring $R$ is \emph{strongly rad-clean} provided that every
element in $R$ is strongly rad-clean. Strongly rad-clean rings
form a natural subclass of strongly clean rings which have stable
range one (cf. [4]). Let $M$ be a right $R$-module, and let
$\varphi\in end_R(M)$. Then we include a relevant diagram to
reinforce the theme of direct sum decompositions:
$$\begin{array}{cccccc}
M&=&A&\bigoplus &B&\\
&&\varphi\downarrow\cong &&\downarrow&\varphi \\
M&=&A&\bigoplus &B&
\end{array}$$
If such diagram holds we call this is an AB-decomposition for
$\varphi$. It turns out by [2, Lemma 40] that $\varphi$ is
strongly $\pi$-regular if and only if there is an AB-decomposition
with $\varphi|_{B}\in N(end(B))$ (the set of nilpotent
elements).\\
Further, $\varphi$ is strongly rad-clean if and only if there is
an \emph{AB-decomposition} with $\varphi|_{B}\in J(end(B))$ (the
Jacobson radical of $end(B)$). Thus, strong rad-cleanness can be
seen as a natural extension of strong $\pi$-regularity. In [2,
Theorem 12], the authors gave a criterion to characterize when a
square matrix over a commutative local ring is strongly clean. We
extend this result to strongly rad-clean matrices over a
commutative local ring. We completely determine when a $2\times 2$
matrix over a commutative local ring is strongly rad-clean. Application to the matrices
over power-series is also studied.

Throughout, all rings are commutative with an identity and all
modules are unitary left modules. Let $M$ be a left $R$-module. We
denote the endomorphism ring of $M$ by $end(M)$ and the
automorphism ring of $M$ by $aut(M)$, respectively. The characteristic polynomial of $A$ is the polynomial
$\chi(A)=det(tI_n-A)$. We always use
$J(R)$ to denote the Jacobson radical and $U(R)$ is the set of invertible elements of a ring $R$. $M_2(R)$
stands for the ring of all $2\times 2$ matrices over $R$, and
$GL_2(R)$ denotes the 2-dimensional general linear group of $R$.

\section{Strongly rad-clean 2x2 matrices over a commutative
local ring}

In this section, we study the structure of strongly rad-clean elements in various situations related
to ordinary ring extensions which have roles in ring theory. We start with a well known characterization of strongly rad-clean element in the endomorphism ring of a module $M$.

\begin{lem} \label{Lemma 1}\ \
Let $E=end(_RM)$, and let $\alpha\in E$. Then the following are equivalent:\end{lem}
\begin{enumerate}
\item [(1)]{\it $\alpha \in E$ is strongly rad-clean.}
\vspace{-.5mm}
\item [(2)]{\it There exists a direct sum decomposition
$M=P\oplus Q$ where $P$ and $Q$ are $\alpha$-invariant, and
$\alpha|_P\in aut(P)$ and $\alpha|_Q\in J\big(end(Q)\big)$.}
\vspace{-.5mm}
\end{enumerate}
\begin{proof} See [6, Proposition 4.1.2].\end{proof}

\begin{lem}\label{Lemma 3} Let $R$ be a ring, let $M$ be a
left $R$-module. Suppose that $x,y,a,b\in end(_RM)$ such that
$xa+yb=1_M, xy=yx=0, ay=ya$ and $xb=bx$. Then $M=ker(x)\oplus
ker(y)$ as left $R$-modules.
\end{lem}
\begin{proof} See [2, Lemma 11].
\end{proof}

A commutative ring $R$ is {\it projective-free} if every finitely generated projective $R$-module is free. Evidently, every commutative local ring is projective-free. We now derive

\begin{lem}\label{Lemma 2} Let $R$ be projective-free. Then $A\in M_2(R)$ is strongly rad-clean if and only if
$A\in GL_2(R)$, or $A\in J\big(M_2(R)\big)$, or $A$ is similar to $diag(\alpha,\beta)$ with $\alpha\in J(R)$ and $\beta\in U(R)$.\end{lem}
\begin{proof}
$\Longrightarrow$ Write $A=E+U, E^2=E, U\in GL_2(R), EA=AE\in J(M_2(R))$. Since $R$ is projective-free,
there exists $P\in GL_n(R)$ such that $PEP^{-1}=diag(0,0), diag(1,1)$ or $diag(1,0)$.
Then $(i)$ $PAP^{-1}=PUP^{-1}$; hence, $A\in GL_2(R)$, $(ii)$ $(PAP^{-1})diag(1,1)=diag(1,1)(PAP^{-1})\in J(M_2(R)$, and so $A\in J(M_2(R))$. $(3)$
$(PAP^{-1})diag(1,0)=diag(1,0)(PAP^{-1})\in J(M_2(R)$ and $PAP^{-1}-diag(1,0)\in GL_2(R)$. Hence, $PAP^{-1}=\left(
\begin{array}{cc}
a&b\\
c&d
\end{array}
\right)$ with $a\in J(R),b=c=0$ and $d\in UR)$. Therefore $A$ is similar to $diag(\alpha,\beta)$ with $\alpha\in J(R)$ and $\beta\in U(R)$.

$\Longleftarrow$ If $A\in GL_2(R)$ or $A\in J\big(M_2(R)\big)$, then $A$ is strongly rad-clean.
We now assume that $A$ is similar to $diag(\alpha,\beta)$ with $\alpha\in J(R)$ and $\beta\in U(R)$. Then $A$ is similar to
$\left(
\begin{array}{cc}
1&0\\
0&0\end{array} \right)+ \left(
\begin{array}{cc}
\alpha -1&0\\
0&\beta\end{array} \right)$ where $$\begin{array}{c}
\left(
\begin{array}{cc}
\alpha -1&0\\
0&\beta\end{array} \right)\in GL_2(R), \left(
\begin{array}{cc}
\alpha&0\\
0&\beta\end{array} \right)\left(
\begin{array}{cc}
1&0\\
0&0\end{array} \right)\in J\big(M_2(R)\big)\\
\left(
\begin{array}{cc}
\alpha -1&0\\
0&\beta\end{array} \right)\left(
\begin{array}{cc}
1&0\\
0&0\end{array} \right)=\left(
\begin{array}{cc}
1&0\\
0&0\end{array} \right)\left(
\begin{array}{cc}
\alpha -1&0\\
0&\beta\end{array} \right).
\end{array}$$ Therefore $A\in M_2(R)$ is strongly rad-clean.\end{proof}

\begin{thm} \label{Theorem 4} Let $R$ be projective-free. Then $A\in M_2(R)$ is strongly rad-clean if and only if
\begin{enumerate}
\item [(1)]{\it $A\in GL_2(R)\big)$, or}
\vspace{-.5mm}
\item [(2)]{\it $A\in J\big(M_2(R)\big)$, or}
\vspace{-.5mm}
\item [(3)]{\it $\chi(A)=0$ has roots $\alpha\in U(R),\beta\in J(R)$.}
\vspace{-.5mm}
\end{enumerate}
\end{thm}
\begin{proof} $\Longrightarrow$ By Lemma \ref{Lemma 2},  $A\in GL_2(R)$, or $A\in
J\big(M_2(R)\big)$, or $A$ is similar to a matrix $\left(
\begin{array}{cc}
\alpha&0\\
0&\beta\end{array} \right)$, where $\alpha\in J(R)$ and $\beta\in U(R)$. Then $\chi(A)=(x-\alpha)(x-\beta)$ has roots
$\alpha\in U(R),\beta\in J(R)$.

$\Longleftarrow$ If $(1)$ or $(2)$ holds, then $A\in M_2(R)$ is
strongly rad-clean. If $(3)$ holds, we assume that $\chi(A)=(t-\alpha)(t-\beta)$. Choose $X=A-\alpha I_2$ and $Y=A-\beta I_2$. Then
$$\begin{array}{c}
X(\beta-\alpha)^{-1}I_2-Y(\beta-\alpha)^{-1}I_2=I_2,\\
XY=YX=0, X(\beta-\alpha)^{-1}I_2=(\beta-\alpha)^{-1}I_2X,\\
(\beta-\alpha)^{-1}I_2Y=Y(\beta-\alpha)^{-1}I_2.
\end{array}$$ By virtue of Lemma \ref{Lemma 3}, we have $2R=ker(X)\oplus ker(Y)$. For any $x\in ker(X)$, we have $(x)AX=(x)XA=0$, and
so $(x)A\in ker(X)$. Then $ker(X)$ is $A$-invariant.
Similarly, $ker(Y)$ is $A$-invariant. For any $x\in ker(X)$, we have $0=(x)X=(x)\big(A-\alpha I_2\big)$; hence,
$(x)A=(x)\alpha I_2$. By hypothesis, we have $A|_{ker(X)}\in J\big(end(ker(X))\big)$. For any $y\in ker(Y)$, we
prove that
$$0=(y)Y=(y)\big(A-\beta I_2\big).$$ This implies that $(y)A=(y)\big(\beta I_2\big)$. Obviously,
$A|_{ker(Y)}\in aut\big(ker(Y)\big)$. Therefore $A\in M_2(R)$ is strongly rad-clean by Lemma \ref{Lemma 1}.
\end{proof}

We have accumulated all the information necessary to prove the following.

\begin{thm}\label{Theorem 5} Let $R$ be a commutative
local ring, and let $A\in M_2(R)$. Then the following are
equivalent:
\begin{enumerate}
\item [(1)]{\it $A\in M_2(R)$ is
strongly rad-clean.} \vspace{-.5mm}
\item [(2)]{\it $A\in GL_2(R)$ or $A\in J\big(M_2(R)\big)$, or $trA\in U(R)$ and the quadratic equation $x^2+x=-\frac{detA}{tr^2A}$ has a root in $J(R)$.}
\vspace{-.5mm}
\item [(3)]{\it $A\in GL_2(R)$ or $A\in J\big(M_2(R)\big)$, or $trA\in U(R), detA\in J(R)$ and the quadratic equation $x^2+x=\frac{detA}{tr^2A-4detA}$ is
solvable.}
\end{enumerate}
\end{thm}
\begin{proof} $(1)\Rightarrow (2)$ Assume that $A\not\in GL_2(R)$ and $A\not\in
J\big(M_2(R)\big)$. By virtue of Theorem \ref{Theorem 4},
$trA\in U(R)$ and the characteristic polynomial $\chi(A)$ has a
root in $J(R)$ and a root in $U(R)$. According to Lemma \ref{Lemma 2},
$A$ is similar to $\left(
\begin{array}{cc}
\lambda&0\\
0&\mu
\end{array} \right)$, where $\lambda\in J(R), \mu\in U(R)$. Clearly,
$y^2-(\lambda+\mu)y+\lambda\mu=0$ has a root $\lambda$ in $J(R)$.
Hence so does the equation
$$(\lambda+\mu)^{-1}y^2-y=-(\lambda+\mu)^{-1}\lambda\mu.$$ Set
$z=(\lambda+\mu)^{-1}y$. Then
$$(\lambda+\mu)z^2-(\lambda+\mu)z=-(\lambda+\mu)^{-1}\lambda\mu.$$ That is,
$z^2-z=-(\lambda+\mu)^{-2}\lambda\mu.$ Consequently,
$z^2-z=-\frac{detA}{tr^2A}$ has a root in $J(R)$. Let $x=-z$. Then $x^2+x=-\frac{detA}{tr^2A}$ has a root in $J(R)$.

$(2)\Rightarrow (3)$ By hypothesis, we prove that the equation
$y^2-y=-\frac{detA}{tr^2A}$ has a root $a\in J(R)$. Assume that $trA\in U(R)$. Then
$\big(a(2a-1)^{-1}\big)^2-\big(a(2a-1)^{-1}\big)=\frac{detA}{tr^2A\cdot
\big(4(a^2-a)+1\big)} =\frac{detA}{tr^2A\cdot
\big(-4(trA)^{-2}detA+1\big)}=\frac{detA}{tr^2A-4detA}. $
Therefore the equation $y^2-y=\frac{detA}{tr^2A-4detA}$ is
solvable. Let $x=-y$. Then $x^2+x=\frac{detA}{tr^2A-4detA}$ is
solvable.

$(3)\Rightarrow (1)$ Suppose $A\not\in GL_2(R)$ and $A\not\in
J\big(M_2(R)\big)$. Then $trA\in U(R), detA\in J(R)$ and the
equation $x^2+x=\frac{detA}{tr^2A-4detA}$ has a root. Let $y=-x$.
Then $y^2-y\frac{detA}{tr^2A-4detA}$ has a root $a\in R$.
Clearly, $b:=1-a\in R$ is a root of this equation. As $a^2-a\in
J(R)$, we see that either $a\in J(R)$ or $1-a\in J(R)$. Thus,
$2a-1=1-2(1-a)\in U(R)$. It is easy to verify that
$\big(a(2a-1)^{-1}trA\big)^2-trA\cdot
\big(a(2a-1)^{-1}trA\big)+detA=-\frac{tr^2A\cdot
(a^2-a)}{4(a^2-a)+1}+detA=0.$ Thus the equation $y^2-trA\cdot
y+detA=0$ has roots $a(2a-1)^{-1}trA$ and $b(2b-1)^{-1}trA$. Since
$ab\in J(R)$, we see that $a+b=1$ and either $a\in J(R)$ or $b\in
J(R)$. Therefore $y^2-trA\cdot y+detA=0$ has a root in $U(R)$ and a
root in $J(R)$. Since $R$ is a commutative local ring, it is projective-free. By virtue of Theorem \ref{Theorem 4}, we obtain the result.\end{proof}

\begin{cor} \label{Corollary 6} Let $R$ be a commutative
local ring, and let $A\in M_2(R)$. Then the following are
equivalent:
\begin{enumerate}
\item [(1)]{\it $A\in M_2(R)$ is
strongly clean.} \vspace{-.5mm}
\item [(2)]{\it $I_2-A\in GL_2(R)$ or $A\in M_2(R)$ is
strongly rad-clean.}
\end{enumerate}
\end{cor}
\begin{proof}
$(2)\Rightarrow (1)$ is trivial.

$(1)\Rightarrow (2)$ In view of [3, Corollary 16.4.33], $A\in
GL_2(R)$, or $I_2-A\in GL_2(R)$ or $trA\in U(R), detA\in J(R)$ and
the quadratic equation \\  $x^2-x=\frac{detA}{tr^2A-4detA}$ is
solvable. Hence $x^2+x=\frac{detA}{tr^2A-4detA}$ is
solvable. According to Theorem \ref{Theorem 5}, we complete the
proof.
\end{proof}

\begin{cor} \label{Corollary 7} Let $R$ be a commutative
local ring. If $\frac{1}{2}\in R$, then the following are
equivalent:
\begin{enumerate}
\item [(1)]{\it $A\in M_2(R)$ is strongly rad-clean.} \vspace{-.5mm}
\item [(2)]{\it $A\in GL_2(R)$ or $A\in J\big(M_2(R)\big)$, or $trA\in U(R), detA\in J(R)$ and $tr^2A-4detA$ is square.}
\vspace{-.5mm}
\end{enumerate}
\end{cor}
\begin{proof} $(1)\Rightarrow (2)$ According to Theorem \ref{Theorem 5}, $A\in GL_2(R)$ or $A\in J\big(M_2(R)\big)$, or $trA\in U(R), detA\in J(R)$ and the quadratic equation $x^2-x=\frac{detA}{tr^2A-4detA}$ is
solvable. If $a\in R$ is the root of the equation, then
$(2a-1)^2=4(a^2-a)+1=\frac{tr^2A}{tr^2A-4detA}\in U(R)$. As in the
proof of Theorem \ref{Theorem 5} , $2a-1\in U(R)$. Therefore
$tr^2A-4detA=\big(trA\cdot (2a-1)^{-1}\big)^2$.

 $(2)\Rightarrow (1)$ If $trA\in U(R), detA\in J(R)$ and $tr^2A-4detA=u^2$ for some $u\in R$, then $u\in U(R)$ and the equation $x^2+x=\frac{detA}{tr^2A-4detA}$ has a root $-\frac{1}{2}u^{-1}(trA+u)$. By virtue of Theorem \ref{Theorem 5}, $A\in M_2(R)$ is strongly rad-clean.
 \end{proof}

 \vskip4mm Every strongly rad-clean matrix over a ring is strongly clean. But there exist strongly clean matrices over a commutative local ring which is not strongly rad-clean as the following shows.

\begin{ex} \label{Example 8} \em Let $R={\mathbb Z}_4$, and let $A=
\left(
\begin{array}{cc}
2&3\\
0&2
\end{array}
\right)\in M_2(R)$. $R$ is a commutative local ring. Then $A=\left(
\begin{array}{cc}
1&0\\
0&1
\end{array}
\right)+\left(
\begin{array}{cc}
1&3\\
0&1
\end{array}
\right)$ is a strongly clean decomposition. Thus $A\in M_2(R)$ is strongly clean.
If $A\in M_2(R)$ is strongly rad-clean, there exist an idempotent $E\in M_2(R)$ and an invertible $U\in M_2(R)$ such that $A=E+U, EA=AE$ and $EAE\in J\big(M_2(R)\big)$. Hence, $AU=A(A-E)=(A-E)A=UA$, and then $E=A-U\in GL_2(R)$ as $A^4=0$. This implies that $E=I_2$, and so $EAE=A\not\in J\big(M_2(R)\big)$, as $J(R)=2R$. This gives a contradiction. Therefore $A\in M_2(R)$ is not strongly rad-clean.
\end{ex}

Following Chen, $a\in R$ is {\it strongly J-clean} provided
that there exists an
idempotent $e\in R$ such that $ea=ae$ and $a-e\in J(R)$ (cf. [3]). Every uniquely
clean ring is strongly J-clean (cf. [9]). We have

\begin{prop} \label{Proposition 16} Every strongly J-clean
element in a ring is strongly rad-clean.
\end{prop}
\begin{proof} For if $a\in R$, there exist $e, w\in R$ such that
$a+1=e+w$ with $e^2=e$, $w\in J(R)$ and $ae=ea$. Multiplying the
last equation from the right and from the left by $e$ we have
$eae=ewe\in eJ(R)e=J(eRe)$. But $a=e+(w-1)$, where $w-1\in U(R)$.
This completes the proof.
\end{proof} 

\vskip4mm Following Cui and Chen, an element $a\in R$ is quaspolar if there exists an idempotent $e\in comm(a)$ such that 
$a+e\in U(R)$ and $ae\in R^{qnil}$. Obviously, $A~\mbox{is strongly J-clean}\Longrightarrow A~\mbox{is strongly rad-clean}\Longrightarrow A~\mbox{is quasipolar}.$ But the converses are not true, as the following shows:

\begin{ex} \label{Example 18}\em $(1)$ Let $R$ be a commutative
local ring and $A= \left(
\begin{array}{cc}
1 &1\\
1&0
\end{array}
\right)$ be in $M_2(R)$. Since $A\in GL_2(R)$, by Lemma \ref{Lemma 2}, it is
strongly rad-clean but is not strongly J-clean, as $I_2-A\not\in J(M_2(R))$.

$(2)$ Let $R={\mathbb Z}_{(3)}$ and $A= \left(
\begin{array}{cc}
2 &1\\
-1&1
\end{array}
\right)$. Then $tr A=3\in J(R)$ and $det A=3\in J(R)$. Hence $A$ is
quasipolar by [5, Theorem 2.6]. Note that $tr A \notin U(R)$,
$A\notin GL_2(R)$ and $A\notin J(M_2(R))$. Thus, $A$ is not
strongly rad-clean, in terms of Corollary \ref{Corollary 7}.
\end{ex}

\vskip4mm Let $R$ be a
commutative local ring. If $A\in M_2(R)$ is strongly rad-clean,
we claim that $A\in GL_2(R)$, or $A\in J\big(M_2(R)\big)$, or $A$ has an
invertible trace. If $A\not\in
GL_2(R)$ and $A\not\in  J\big(M_2(R)\big)$, it follows from Lemma
\ref{Lemma 2} that $A$ is similar to a matrix $\left(
\begin{array}{cc}
1+w_1&0\\
0&w_2\end{array} \right)$ where $1+w_1\in J(R), w_2\in U(R)$.
Therefore $tr(A)=(1+w_1)+w_2\in U(R)$, and we are done.

We set $B_{12}(a)=\left(
\begin{array}{cc}
1&a\\
0&1 \end{array} \right)$ and $B_{21}(a)=\left(
\begin{array}{cc}
1&0\\
a&1 \end{array} \right)$. We have

\begin{thm} \label{Theorem 10} Let $R$ be a commutative local
ring. Then the following are equivalent:
\begin{enumerate}
\item [(1)]{\it Every $A\in M_2(R)$ with invertible trace is strongly rad-clean.} \vspace{-.5mm}
\item [(2)]{\it For any $\lambda\in J(R), \mu\in U(R)$, the quadratic equation $x^2+\mu x+\lambda=0$ is solvable.}
\vspace{-.5mm}
\end{enumerate}
\end{thm}
\begin{proof} $(1)\Rightarrow (2)$ Let $\lambda\in J(R), \mu\in U(R)$. Choose $A=\left(
\begin{array}{cc}
0&\lambda\\
1&-\mu\end{array} \right)$. Then $A\in M_2(R)$ is strongly rad
clean. Obviously, $A\not\in GL_2(R)$ and $A\not\in
J\big(M_2(R)\big)$. In view of Theorem \ref{Theorem 4} the quadratic
polynomial $\chi(A)=x^2+\mu x+\lambda$ is solvable.

$(2)\Rightarrow (1)$ Let $A= \left(
\begin{array}{cc}
a&b\\
c&d \end{array} \right)$ with $tr(A)\in U(R)$. Case I. $c\in
U(R)$. Then
$$diag(c,1)B_{12}(-ac^{-1})AB_{12}(ac^{-1})diag(c^{-1},1)=\left(
\begin{array}{cc}
0&\lambda\\
1&\mu\end{array} \right)$$ for some $\lambda,\mu\in R$. If
$\lambda\in U(R)$, then $A\in GL_2(R)$, and so it is strongly rad-clean.
If $\lambda\in J(R)$, then $\mu\in U(R)$. Then $A$ is strongly rad-clean by Theorem 2.5.\\ Case II. $b\in U(R)$. Then
$$\left(
\begin{array}{cc}
0&1\\
1&0\end{array} \right)A\left(
\begin{array}{cc}
0&1\\
1&0\end{array} \right)=\left(
\begin{array}{cc}
d&c\\
b&a\end{array} \right),$$ and the result follows from Case I.\\
Case III. $c, b\in J(R), a-d\in U(R)$. Then
$$B_{21}(-1)AB_{21}(1)= \left( \begin{array}{cc}
                           a+b & b \\ c-a+d-b & d-b \\
                         \end{array} \right)$$ where $a-d+b-c\in U(R)$; hence
the result follows from Case I. \\Case IV. $c, b\in J(R), a, d\in
U(R)$. Then
$$B_{21}(-ca^{-1})A= \left(
\begin{array}{cc}
a&b\\
0&d-ca^{-1}b\end{array} \right);$$ hence, $A\in GL_2(R)$.\\ Case
V. $c, b, a, d\in J(R)$. Then $A\in J\big(M_2(R)\big)$, and so
$tr(A)\in J(R)$, a contradiction.

Therefore $A\in M_2(R)$ with invertible trace is strongly rad-clean.\end{proof}

\begin{ex} \label{Example 11} \em Let $R={\mathbb Z}_4$. Then $R$ is a
commutative local ring. For any $\lambda\in J(R), \mu\in U(R)$, we
directly check that the quadratic equation $x^2+\mu x+\lambda=0$
is solvable. Applying Lemma \ref{Theorem 10}, every $2\times 2$ matrix over
$R$ with invertible trace is strongly rad-clean. In this case, $M_2(R)$ is not strongly rad-clean.\end{ex}

\begin{ex} \label{Example 12} Let $R=\widehat{{\mathbb Z}}_2$ be the
ring of 2-adic integers. Then every $2\times 2$ matrix with
invertible trace is strongly rad-clean.
\end{ex}
\begin{proof} Obviously, $R$ is a commutative local ring. Let
$\lambda\in J(R), \mu\in U(R)$. Then $\left(
\begin{array}{cc}
0&-\lambda\\
1&-\mu
\end{array}
\right)\in M_2(R)$ is strongly clean, by [5, Theorem 3.3]. Clearly, $det(A)=\lambda\in J(R)$. As $R/J(R)\cong {\mathbb Z}_2$, we see that
$\mu\in -1+J(R)$, and then $det(A-I_2)=\lambda+\mu+1\in J(R)$. In light of [5, Lemma 3.1],
the equation $x^2+\mu x+\lambda=0$ is solvable. This completes the proof, by Theorem \ref{Theorem 10}.
\end{proof}

\vskip4mm We note that matrix with non-invertible trace over commutative local rings maybe not strongly rad-clean. For instance, $A=\left(
\begin{array}{cc}
1&1\\
1&1\end{array} \right)\in M_2(\widehat{{\mathbb Z}}_2)$ is not
strongly rad-clean.

\section{Strongly rad-clean matrices over power series over
commutative local rings}

 We now investigate strongly rad-clean
matrices over power series over commutative local rings.

\begin{lem} \label{Lemma 13} Let $R$ be a commutative ring,
and let $A(x_1,\cdots,x_n)\in M_2\big(R[[x_1,$ $\cdots,x_n]]\big)$. Then the following
hold:
\begin{enumerate}
\item [(1)]{\it $A(x_1,\cdots,x_n)\in GL_2\big(R[[x_1,\cdots,x_n]])$ if and only if $A(0,\cdots ,0)\in GL_2(R)$.}
\vspace{-.5mm}
\item [(2)]{\it $A(x_1,\cdots,x_n)\in J\big(M_2(R[[x_1,\cdots,x_n]])\big)$ if and only if $A(0,\cdots ,0)\in J\big(M_2(R)\big)$.}
\vspace{-.5mm}
\end{enumerate}
\end{lem}
\begin{proof} $(1)$ We suffice to prove for $n=1$. If $A(x_1)\in
GL_2\big(R[[x_1]])$, it is easy to verify that $A(0)\in GL_2(R)$.
Conversely, assume that $A(0)\in GL_2(R)$. Write $$A(x_1)= \left(
\begin{array}{cc}
\sum\limits_{i=0}^{\infty}a_ix_1^i&\sum\limits_{i=0}^{\infty}b_ix_1^i\\
\sum\limits_{i=0}^{\infty}c_ix_1^i&\sum\limits_{i=0}^{\infty}d_ix_1^i
\end{array}
\right),$$ where $A(0)= \left(
\begin{array}{cc}
a_0 &b_0 \\
c_0&d_0
\end{array}
\right).$ We note that the
determinant of $A(x_1)$ is $a_0d_0-c_0b_0+x_1f(x_1)$, which is a unit plus an
element of the radical of $R[[x_1]]$. Thus, $A(x_1)\in
GL_2\big(R[[x_1]])$, as required.

$(2)$ It is immediate from $(1)$.
\end{proof}

\begin{thm} \label{Theorem 14} Let $R$ be a commutative local
ring, and let $A(x_1,\cdots,x_n)\in M_2\big(R[[x_1,$ $\cdots,x_n]])$. Then the following are
equivalent:
\begin{enumerate}
\item [(1)]{\it $A(x_1,\cdots,x_n)\in M_2\big(R[[x_1,\cdots,x_n]])$ is strongly rad-clean.}
\vspace{-.5mm}
\item [(2)]{\it $A(x_1,\cdots,x_n)\in M_2\big(R[[x_1,\cdots,x_n]]/(x_1^{m_1}\cdots x_n^{m_n})\big)$ is strongly rad-clean.}\vspace{-.5mm}
\item [(3)]{\it $A(0,\cdots ,0)\in M_2(R)$ is strongly rad-clean.}
\vspace{-.5mm}
\end{enumerate}
\end{thm}
\begin{proof} $(1)\Rightarrow (2)$ and $(2)\Rightarrow (3)$ are obvious.

$(3)\Rightarrow (1)$ It will suffice to prove for $n=1$.
Set $x=x_1$. Clearly, $R[[x]]$ is a commutative local
ring. Since $A(0)$ is strongly clean in $M_2(R)$, it follows from
Theorem \ref{Theorem 4} that $A(0)\in GL_2(R)$, or $A(0)\in
J\big(M_2(R)\big)$, or $\chi\big(A(0)\big)$ has a root $\alpha\in
J(R)$ and a root $\beta\in U(R)$. If  $A(0)\in GL_2(R)$ or
$A(0)\in J\big(M_2(R)\big)$, in view of Lemma \ref{Lemma 13},  $A(x)\in
GL_2\big(R[[x]])$ or $A(x)\in J\big(M_2(R[[x]])\big)$. Hence, $A(x)\in
M_2\big(R[[x]]\big)$ is strongly rad-clean. Thus, we may assume
that $\chi\big(A(0)\big)=t^2+\mu t+\lambda$ has a root $\alpha\in
J(R)$ and a root $\beta\in U(R)$.

Write $\chi\big(A(x)\big)=t^2+\mu(x)t+\lambda(x)$ where
$\mu(x)=\sum\limits_{i=0}^{\infty}\mu_ix^i,
\lambda(x)=\sum\limits_{i=0}^{\infty}\lambda_ix^i\in R[[x]]$ and
$\mu_0=\mu,\lambda_0=\lambda$. Let $b_0=\alpha$. It is easy to verify that $\mu_0=\alpha+\beta\in U(R)$.
Hence, $2b_0+\mu_0\in U(R)$. Choose
$$\begin{array}{c}
b_1=(2b_0+\mu_0)^{-1}(-\lambda_1-\mu_1b_0),\\
b_2=(2b_0+\mu_0)^{-1}(-\lambda_2-\mu_1b_1-\mu_2b_0-b_1^2),\\
\vdots
\end{array}$$ Then $y=\sum\limits_{i=0}^{\infty}b_ix^i\in R[[x]]$
is a root of $\chi\big(A(x)\big)$. In addition, $y\in
J\big(R[[x]]\big)$ as $b_0\in J(R)$. Since $y^2+\mu(x)y+\lambda(x)=0$, we have
$\chi\big(A(x)\big)=(t-y)(t+y)+\mu (t-y)=(t-y)(t+y+\mu)$. Set
$z=-y-\mu$. Then $z\in U\big(R[[x]]\big)$ as $\mu\in
U\big(R[[x]]\big)$. Therefore $\chi\big(A(x)\big)$ has a root in
$J\big(R[[x]]\big)$ and a root in $U\big(R[[x]]\big)$. According
to Theorem \ref{Theorem 4}, $A(x)\in M_2\big(R[[x]])$ is strongly rad-clean,
as asserted.
\end{proof}

\begin{cor} \label{Corollary 15} Let $R$ be a commutative
local ring. Then the following are equivalent:
\begin{enumerate}
\item [(1)]{\it Every $A\in M_2(R)$ with invertible trace is strongly rad-clean.} \vspace{-.5mm}
\item [(2)]{\it Every $A(x_1,\cdots,x_n)\in M_2\big(R[[x_1,\cdots,x_n]]\big)$ with invertible trace is strongly rad-clean.}
\vspace{-.5mm}
\end{enumerate}
\end{cor}
\begin{proof} $(1)\Rightarrow (2)$ Let $A(x_1,\cdots,x_n)\in M_2\big(R[[x_1,\cdots,x_n]]\big)$ with invertible
trace. Then $trA(0,\cdots ,0)\in U(R)$. By hypothesis, $A(0,\cdots ,0)\in M_2(R)$ is
strongly rad-clean. In light of Theorem \ref{Theorem 4}, $A(x_1,\cdots,x_n)\in
M_2\big(R[[x_1,\cdots,x_n]])$ is strongly rad-clean.

$(2)\Rightarrow (1)$ is obvious.
\end{proof}

\vskip10mm ACKNOWLEDGEMENT

\vskip4mm The authors are grateful to the referee for his/her
helpful comments and suggestions which led to a much improved paper.

\end{document}